\documentclass[11pt,reqno]{amsart}

\newtheorem{proposition}{Proposition}[section]

\newtheorem{corollary}[proposition]{Corollary}
\newtheorem{theorem}[proposition]{Theorem}

\theoremstyle{definition}
\newtheorem{definition}[proposition]{Definition}

\newtheorem{examples}[proposition]{Examples}
\newtheorem{remark}[proposition]{Remark}

\newcommand{\thlabel}[1]{\label{th:#1}}
\newcommand{\thref}[1]{Theorem~\ref{th:#1}}
\newcommand{\selabel}[1]{\label{se:#1}}

\newcommand{\prlabel}[1]{\label{pr:#1}}
\newcommand{\prref}[1]{Proposition~\ref{pr:#1}}
\newcommand{\colabel}[1]{\label{co:#1}}
\newcommand{\coref}[1]{Corollary~\ref{co:#1}}
\newcommand{\relabel}[1]{\label{re:#1}}

\newcommand{\exlabel}[1]{\label{ex:#1}}

\newcommand{\delabel}[1]{\label{de:#1}}

\newcommand{\eqlabel}[1]{\label{eq:#1}}
\newcommand{\equref}[1]{(\ref{eq:#1})}

\def\ot{\otimes}

\newcommand{\Cc}{\mathcal{C}}

\def\*C{{}^*\hspace*{-1pt}{\Cc}}
\def\text#1{{\rm {\rm #1}}}

\input xy
\xyoption {all} \CompileMatrices

\usepackage{amssymb}
\usepackage{color,amssymb,graphicx,amscd,amsmath}
\usepackage[colorlinks,urlcolor=blue,linkcolor=blue,citecolor=blue]{hyperref}

\begin{document}
%\title[The universal bialgebra of a finite dimensional Leibniz algebras]
%{The universal bialgebra and the automorphisms group of a finite
%dimensional Leibniz algebras}

\title[A new invariant for finite dimensional Leibniz algebras]
{A new invariant for finite dimensional Leibniz/Lie algebras}

\author{A. L. Agore}
\address{Vrije Universiteit Brussel, Pleinlaan 2, B-1050 Brussels, Belgium}
\address{Simion Stoilow Institute of Mathematics of the Romanian Academy, P.O. Box 1-764, 014700 Bucharest, Romania}
\email{ana.agore@vub.be,\, ana.agore@gmail.com}

\author{G. Militaru}
\address{Faculty of Mathematics and Computer Science, University of Bucharest, Str.
Academiei 14, RO-010014 Bucharest 1, Romania}
\email{gigel.militaru@fmi.unibuc.ro and gigel.militaru@gmail.com}

\thanks{The first named author was supported by a grant of Romanian Ministery
of Research and Innovation, CNCS - UEFISCDI, project number
PN-III-P1-1.1-TE-2016-0124, within PNCDI III and is a fellow of
FWO (Fonds voor Wetenschappelijk Onderzoek -- Flanders).}


%\maketitle

\maketitle

\begin{abstract}
For an $n$-dimensional Leibniz/Lie
algebra $\mathfrak{h}$ over a field $k$ we introduce a new
invariant ${\mathcal A}(\mathfrak{h})$, called the \emph{universal
algebra} of $\mathfrak{h}$, as a quotient of the polynomial
algebra $k[X_{ij} \, | \, i, j = 1, \cdots, n]$ through an ideal
generated by $n^3$ polynomials. We prove that ${\mathcal
A}(\mathfrak{h})$ admits a unique bialgebra structure which makes
it an initial object among all commutative bialgebras coacting on
$\mathfrak{h}$. The new object ${\mathcal A} (\mathfrak{h})$ is
the key tool in answering two open problems in Lie
algebra theory. First, we prove that the automorphism group ${\rm
Aut}_{Lbz} (\mathfrak{h})$ of $\mathfrak{h}$ is isomorphic to the
group $U \bigl( G({\mathcal A} (\mathfrak{h})^{\rm o} ) \bigl)$ of
all invertible group-like elements of the finite dual ${\mathcal
A} (\mathfrak{h})^{\rm o}$. Secondly, for an abelian group $G$, we
show that there exists a bijection between the set of all
$G$-gradings on $\mathfrak{h}$ and the set of all bialgebra
homomorphisms ${\mathcal A} (\mathfrak{h}) \to k[G]$. Based on
this, all $G$-gradings on $\mathfrak{h}$ are explicitly classified
and parameterized. ${\mathcal A} (\mathfrak{h})$ is also used to
prove that there exists a universal commutative Hopf algebra
associated to any finite dimensional Leibniz algebra
$\mathfrak{h}$.
\end{abstract}

\section*{Introduction}
Let $A$ be a unital associative algebra over a field $k$. M.E.
Sweedler's result \cite[Theorem 7.0.4]{Sw} which states that the
functor ${\rm Hom} (- , \, A) : {\rm CoAlg}_k \to {\rm Alg}_k
^{\rm op}$ from the category of coalgebras over $k$ to the
opposite of the category of $k$-algebras has a right adjoint
denoted by ${\rm M} (-, \, A)$ proved itself remarkable through
its applications. Furthermore, ${\rm M} (A, \, A)$ turns out to be
a bialgebra and the final object in the category of all bialgebras
that act on $A$ through a module algebra structure. The dual
version was considered by Tambara \cite{Tambara} and in a special
(graded) case by Manin \cite{Manin}. To be more precise,
\cite[Theorem 1.1]{Tambara} proves that if $A$ is a finite
dimensional algebra, then the tensor functor $A \ot - : {\rm
Alg}_k \to {\rm Alg}_k$ has a left adjoint denoted by $a(A, \,
-)$. In the same spirit, $a(A,\, A)$ is proved to be a bialgebra as well and the
initial object in the category of all bialgebras that coact on $A$
through a comodule algebra structure. Both objects are very
important: as explain in \cite{Manin}, the Hopf envelope of $a(A,
A)$ plays the role of a symmetry group in non-commutative
geometry. For further details we refer to \cite{AA, ana2019, anagorj}.
A more general construction, which contains all the above as
special cases, was recently considered in \cite{anagorj2} in the
context of $\Omega$-algebras.

The starting point of this paper was an attempt to prove the
counterpart of Tambara's result at the level of Leibniz algebras.
Introduced by Bloh \cite{Bl1} and rediscovered by Loday
\cite{Lod2}, Leibniz algebras are non-commutative generalizations
of Lie algebras. This new concept generated a lot of interest
mainly due to its interaction with (co)homology theory, vertex
operator algebras, the Godbillon-Vey invariants for foliations or
differential geometry. Another important concept for our approach
is that of a \emph{current Lie} algebra. Being first introduced in
physics \cite{Gel}, current Lie algebras, are Lie algebras of the
form $\mathfrak{g} \ot A$ with the bracket given by $\left[
x\ot a, \, y\ot b \right] := \left[x, \, y\right] \ot ab$, for all
$x$, $y \in \mathfrak{g}$ and $a$, $b\in A$, where $\mathfrak{g}$ is a Lie algebra and
$A$ is a commutative algebra. They are interesting
objects that arise in various branches of mathematics and physics
such as the theory of affine Kac-Moody algebras or the structure
of modular semisimple Lie algebras (see \cite{Zus, Zus2}). Current
Leibniz algebras are immediate generalizations, i.e. Leibniz algebras of the form $\mathfrak{h}\ot A $ whose bracket is
defined as in the case of Lie algebras, where this time
$\mathfrak{h}$ is a Leibniz algebra and $A$ a commutative algebra.
By fixing a Leibniz algebra $\mathfrak{h}$, we obtain a functor
$\mathfrak{h}\ot - : {\rm ComAlg}_k \to {\rm Lbz}_k$ from the
category of commutative algebras to the category of Leibniz
algebras called the current Leibniz algebra functor.
\thref{adjunctie} proves that the functor $\mathfrak{h} \ot - \, :
{\rm ComAlg}_k \to {\rm Lbz}_k$ has a left adjoint, denoted by
${\mathcal A} (\mathfrak{h}, \, -)$, if and only if $\mathfrak{h}$
is finite dimensional. For an $n$-dimensional Leibniz algebra
$\mathfrak{h}$ and an arbitrary Leibniz algebra $\mathfrak{g}$
with $|I| = {\rm dim}_k (\mathfrak{g})$, ${\mathcal A}
(\mathfrak{h}, \, \mathfrak{g})$ is a quotient of the usual
polynomial algebra $k [X_{si} \, | s = 1, \cdots, n, \, i\in I]$.
The commutative algebra ${\mathcal A} (\mathfrak{h}, \,
\mathfrak{g})$ is a very powerful
tool for studying Leibniz/Lie algebras as it captures most of the 
essential information on the two Leibniz/Lie algebras. Note for
instance that the characters of this algebra parameterize the set
of all Leibniz algebra homomorphisms between $\mathfrak{g}$ and
$\mathfrak{h}$ (\coref{morlbz}). \thref{adjunctie} has obviously a
Lie algebra counterpart. In this case, if $\mathfrak{g}$ is a Lie
algebra and $m$ a positive integer then the characters of the
commutative algebra ${\mathcal A} (\mathfrak{gl} (m, k), \,
\mathfrak{g})$ parameterize the space of all $m$-dimensional
representations of $\mathfrak{g}$ (\coref{replie}). The
commutative algebra ${\mathcal A} (\mathfrak{h}) := {\mathcal A}
(\mathfrak{h}, \, \mathfrak{h})$ is called the \emph{universal
algebra} of $\mathfrak{h}$: it is a quotient of the polynomial
algebra $M(n) := k[X_{ij} \, | \, i, j = 1, \cdots, n]$ through an
ideal generated by $n^3$ polynomials called the \emph{universal
polynomials} of $\mathfrak{h}$. \prref{bialgebra} proves that
${\mathcal A} (\mathfrak{h})$ has a canonical bialgebra structure
such that the projection $\pi: M(n) \to {\mathcal A}
(\mathfrak{h})$ is a bialgebra homomorphism. The first main application of
the universal (bi)algebra ${\mathcal A} (\mathfrak{h})$ of
$\mathfrak{h}$ is given in \thref{automorf} which provides an
explicit description of a group isomorphism between the group of
automorphisms of $\mathfrak{h}$ and the group of all invertible
group-like elements of the finite dual ${\mathcal A}
(\mathfrak{h})^{\rm o}$:
$$
{\rm Aut}_{{\rm Lbz}} (\mathfrak{h}) \cong U \bigl(G\bigl(
{\mathcal A} (\mathfrak{h})^{\rm o} \bigl) \bigl).
$$
We mention that achieving a complete description of the
automorphisms group ${\rm Aut}_{\rm Lie} (\mathfrak{h})$ of a
given Lie algebra $\mathfrak{h}$ is a classical \cite{borel,
jacobson} and notoriously difficult problem intimately related to
the structure of Lie algebras (for more details see the recent
papers \cite{am-2018, Ar, fisher} and their references). The unit
of the adjunction depicted in \thref{adjunctie}, denoted by
$\eta_{\mathfrak{h}} : \mathfrak{h} \to \mathfrak{h} \ot {\mathcal
A} (\mathfrak{h})$, endows $\mathfrak{h}$ with a right ${\mathcal
A}(\mathfrak{h})$-comodule structure and the pair $({\mathcal A}
(\mathfrak{h}), \, \eta_{\mathfrak{h}})$ is the intial object in
the category of all commutative bialgebras that coact on the
Leibniz algebra $\mathfrak{h}$ (\thref{univbialg}). This result
allows for two important consequences: \prref{graduari} proves
that for an abelian group $G$ there exists an explicitly described
bijection between the set of all $G$-gradings on $\mathfrak{h}$
and the set of all bialgebra homomorphisms ${\mathcal A}
(\mathfrak{h}) \to k[G]$.
Furthermore, all $G$-gradings on
$\mathfrak{h}$ are classified in
\thref{nouaclas}: the set $G$-${\rm 
\textbf{gradings}}(\mathfrak{h})$ of isomorphism classes of all
$G$-gradings on $\mathfrak{h}$ is in bijection with the quotient
set $ {\rm Hom}_{\rm BiAlg} \, \bigl( {\mathcal A} (\mathfrak{h})
, \, k[G] \bigl)/\approx$ of all bialgebra homomorphisms ${\mathcal A}
(\mathfrak{h}) \, \to  k[G]$ by the equivalence relation
given by the usual conjugation with an invertible group-like
element. Secondly, if $G$ is a finite group, \prref{actiuni}
shows that there exists a bijection between the set of all actions
as automorphisms of $G$ on $\mathfrak{h}$ (i.e. morphisms of
groups $G \to {\rm Aut}_{{\rm Lbz}} (\mathfrak{h})$) and the set
of all bialgebra homomorphisms ${\mathcal A} (\mathfrak{h}) \to
k[G]^*$. Concerning the last two results we mention that there
exists a vast literature concerning the classification of all
$G$-gradings on a given Lie algebra (see \cite{bahturin, eld, G, MZ}
and their references). On the other hand, the study of actions as
automorphisms of a group $G$ on a Lie algebra $\mathfrak{h}$ goes
back to Hilbert's invariant theory whose foundation was set at the
level of Lie algebras in the classical papers \cite{borel, bra,
thrall}; for further details see \cite{am-2018} and the references
therein. Using once again \thref{univbialg} and the existence of a
free commutative Hopf algebra on any commutative bialgebra
\cite[Theorem 65, (2)]{T}, we prove in \thref{univhopf} that there
exists a universal coacting Hopf algebra on any finite dimensional
Leibniz algebra. We point out that, to the best of our knowledge,
this is the only universal Hopf algebra associated to a Leibniz
algebra appearing in the literature.

\section{Preliminaries}\selabel{prel}
All vector spaces, (bi)linear maps, Leibniz, Lie or associative
algebras, bialgebras and so on are over an arbitrary field $k$ and $\ot = \ot_k$. A
Leibniz algebra is a vector space $\mathfrak{h}$, together with a
bilinear map $[- , \, -] : \mathfrak{h} \times \mathfrak{h} \to
\mathfrak{h}$ satisfying the Leibniz identity for any $x$, $y$, $z
\in \mathfrak{h}$:
\begin{equation}\eqlabel{Lbz1}
\left[ x,\, \left[y, \, z \right] \right] = \left[ \left[x, \,
y\right], \, z \right] - \left[\left[x, \, z\right] , \, y\right]
\end{equation}
Any Lie algebra is a Leibniz algebra, and a Leibniz algebra
$\mathfrak{h}$ satisfying $[x, \, x] = 0$, for all $x \in
\mathfrak{h}$ is a Lie algebra. We shall denote by ${\rm Aut}_{\rm
Lbz} (\mathfrak{h})$ (resp. ${\rm Aut}_{\rm Lie} (\mathfrak{h})$)
the automorphisms group of a Leibniz (resp. Lie) algebra
$\mathfrak{h}$. Any vector space $V$ is a Leibniz algebra with trivial bracket $[x,\, y] := 0$, for all $x$, $y\in V$ -- such a
Leibniz algebra is called \emph{abelian} and will be denoted by
$V_0$. For two subspaces $A$ and $B$ of a Leibniz algebra
$\mathfrak{h}$ we denote by $[A, \, B]$ the vector space generated
by all brackets $[a, \, b]$, for any $a \in A$ and $b\in B$. In
particular, $\mathfrak{h}' := [\mathfrak{h}, \, \mathfrak{h}]$ is
called the derived subalgebra of $\mathfrak{h}$.

We shall denote by ${\rm Lbz}_k$, ${\rm Lie}_k$ and ${\rm
ComAlg}_k$ the categories of Leibniz, Lie and respectively
commutative associative algebras. Furthermore, the category of
commutative bialgebras (resp. Hopf algebras) is denoted
by ${\rm ComBiAlg}_k$ (resp. ${\rm ComHopf}_k$).
If $\mathfrak{h}$ is a
Leibniz algebra and $A$ a commutative algebra then
$\mathfrak{h}\ot A $ is a Leibniz algebra with bracket defined
for any $x$, $y \in \mathfrak{h}$ and $a$, $b\in A$ by:
\begin{equation}\eqlabel{curant}
\left[ x\ot a, \, y\ot b \right] := \left[x, \, y\right] \ot ab
\end{equation}
called the \emph{current Leibniz} algebra. Indeed, as $A$
is a commutative and associative algebra, we have:
\begin{eqnarray*}
&& \left[ \, \left[x \ot a, \, y\ot b \right], \, z\ot c \right] -
\left[ \, \left[x \ot a, \, z\ot c \right], \, y\ot b \right] \\
&& = \left[ \, \left[x, \, y\right], \, z \right] \ot abc -
\left[ \, \left[x, \, z\right] , \, y\right] \ot acb \\
&& = \bigl( \left[ \, \left[x, \, y\right], \, z \right] - \left[
\, \left[x, \, z\right] , \, y\right] \bigl) \ot abc \\
&& = \left[x,\, \left[y, \, z \right] \, \right] \ot abc =
\left[x\ot a, \, \left[y\ot b, \, z \ot c\right] \, \right]
\end{eqnarray*}
for all $x$, $y$, $z \in \mathfrak{h}$ and $a$, $b$, $c\in A$,
i.e. the Leibniz identity \equref{Lbz1} holds for $\mathfrak{h}\ot
A $. For a fixed Leibniz algebra $\mathfrak{h}$,  assigning $A
\mapsto \mathfrak{h} \ot A$ defines a functor $\mathfrak{h} \ot - \, : {\rm ComAlg}_k \to {\rm Lbz}_k$ from the
category of commutative $k$-algebras to the category of Leibniz
algebras called the current Leibniz algebra functor. If $f : A \to
B$ is an algebra map then ${\rm Id}_{\mathfrak{h}} \, \ot f :
\mathfrak{h} \ot A \to \mathfrak{h} \ot B$ is a morphism of
Leibniz algebras.

For basic concepts on category theory we refer the reader to \cite{mlane}
and for unexplained notions pertaining to Hopf algebras to \cite{radford, Sw}.

\section{Universal constructions}\selabel{sect2}

Our first resut is the Leibniz algebra counterpart of
\cite[Theorem 1.1]{Tambara}.

\begin{theorem}\thlabel{adjunctie}
Let $\mathfrak{h}$ be a Leibniz algebra. Then the current Leibniz
algebra functor $\mathfrak{h} \ot - \, : {\rm ComAlg}_k \to {\rm
Lbz}_k$ has a left adjoint if and only if $\mathfrak{h}$ is finite
dimensional. Moreover, if $\mathfrak{h} \neq 0$ the functor
$\mathfrak{h} \ot - \, $ does not admit a right adjoint.
\end{theorem}

\begin{proof}
Assume first that $\mathfrak{h}$ is a finite dimensional
Leibniz algebra and ${\rm dim}_k (\mathfrak{h}) = n$. Fix
$\{e_1, \cdots, e_n\}$ a basis in $\mathfrak{h}$ and let
$\{\tau_{i, j}^s \, | \, i, j, s = 1, \cdots, n \}$ be the
structure constants of $\mathfrak{h}$, i.e. for any $i$, $j = 1,
\cdots, n$ we have:
\begin{equation}\eqlabel{const1}
\left[e_i, \, e_j \right]_{\mathfrak{h}} = \sum_{s=1}^n \,
\tau_{i, j}^s \, e_s.
\end{equation}
In what follows we shall explicitly construct a left adjoint of the current Leibniz algebra functor
$\mathfrak{h} \ot -$, denoted by ${\mathcal A}
(\mathfrak{h}, \, - ) : {\rm Lbz}_k \to {\rm ComAlg}_k$. Let $\mathfrak{g}$ be a Leibniz algebra and
let $\{f_i \, | \, i \in I\}$ be a basis of $\mathfrak{g}$. For
any $i$, $j\in I$, let $B_{i,j} \subseteq I$ be a finite subset of
$I$ such that for any $i$, $j \in I$ we have:
\begin{equation}\eqlabel{const2}
\left[f_i, \, f_j \right]_{\mathfrak{g}} = \sum_{u \in B_{i, j}}
\, \beta_{i, j}^u \, f_{u}.
\end{equation}
Let $k [X_{si} \, | s = 1, \cdots, n, \, i\in I]$
be the usual polynomial algebra and define
\begin{equation}\eqlabel{alguniv}
{\mathcal A} (\mathfrak{h}, \, \mathfrak{g}) :=  k [X_{si} \, | s
= 1, \cdots, n, \, i\in I] / J
\end{equation}
where $J$ is the ideal generated by all polynomials of the form
\begin{equation}\eqlabel{poluniv}
P_{(a, i, j)} ^{(\mathfrak{h}, \, \mathfrak{g})} := \sum_{u \in
B_{i, j}} \, \beta_{i, j}^u \, X_{au} - \sum_{s, t = 1}^n \,
\tau_{s, t}^a \, X_{si} X_{tj}, \quad {\rm for}\,\, {\rm all}\,\, a = 1, \cdots, n\,\, {\rm and}\,\, i,\, j\in I.
\end{equation}
We denote by $x_{si}
:= \widehat{X_{si}}$ the class of ${X_{si}}$ in the
algebra ${\mathcal A} (\mathfrak{h}, \, \mathfrak{g})$; thus the following relations hold in
the commutative algebra ${\mathcal A} (\mathfrak{h}, \,
\mathfrak{g})$:
\begin{equation}\eqlabel{relatii}
\sum_{u \in B_{i, j}} \, \beta_{i, j}^u \, x_{au} = \sum_{s, t =
1}^n \, \tau_{s, t}^a \, x_{si} x_{tj}, \quad {\rm for}\,\, {\rm all}\,\, a = 1, \cdots, n,\,\, {\rm and} \,\, i,\, j\in I.
\end{equation}
Now we consider the map:
\begin{equation}\eqlabel{unitadj}
\eta_{\mathfrak{g}} : \mathfrak{g} \to \mathfrak{h} \ot {\mathcal
A} (\mathfrak{h}, \, \mathfrak{g}), \quad \eta_{\mathfrak{g}}
(f_i) := \sum_{s=1}^n \, e_s \ot x_{si},  \quad {\rm for\,\,
all}\,\, i\in I.
\end{equation}

We shall prove first that $\eta_{\mathfrak{g}}$
is a Leibniz algebra homomorphism. Indeed, for any $i$, $j\in I$
we have:
\begin{eqnarray*}
&& \left[\eta_{\mathfrak{g}} (f_i), \, \eta_{\mathfrak{g}} (f_j)
\right]_{\mathfrak{h} \ot {\mathcal A} (\mathfrak{h}, \,
\mathfrak{g})} = \left[ \sum_{s=1}^n \, e_s \ot x_{si}, \,\,
\sum_{t=1}^n \, e_t \ot x_{tj} \right]_{\mathfrak{h} \ot {\mathcal
A} (\mathfrak{h}, \, \mathfrak{g})}= \sum_{s, t =1}^n \, \left[e_s, \, e_t \right]_{\mathfrak{h}}
\ot x_{si} x_{tj}\\
&& = \sum_{a =1}^n \, e_a \ot \underline{\bigl(\sum_{s, \,t =
1}^n \, \tau_{s, t}^a \, x_{si} x_{tj} \bigl)} \,\, \stackrel{\equref{relatii}}
= \,\, \sum_{a=1}^n \, e_a \ot \bigl( \sum_{u \in B_{i, j}} \,
\beta_{i, j}^u \, x_{au}\bigl) = \sum_{u \in B_{i, j}} \, \beta_{i, j}^u \, \eta_{\mathfrak{g}}
(f_u) \\
&& = \eta_{\mathfrak{g}} (\left[f_i, \, f_j \right]_{\mathfrak{g}})
\end{eqnarray*}

Now we prove that for any Leibniz algebra $\mathfrak{g}$ and any
commutative algebra $A$ the map defined below is bijective:
\begin{equation}\eqlabel{adjp}
\gamma_{\mathfrak{g}, \, A} \, : {\rm Hom}_{\rm Alg_k} \, (
{\mathcal A} (\mathfrak{h}, \, \mathfrak{g}), \, A) \to {\rm
Hom}_{\rm Lbz_k} \, (\mathfrak{g}, \, \mathfrak{h} \ot A), \quad
\gamma_{\mathfrak{g}, \, A} (\theta) := \bigl( {\rm
Id}_{\mathfrak{h}} \ot \theta \bigl) \circ \eta_{\mathfrak{g}}
\end{equation}
To this end, let $f : \mathfrak{g} \to \mathfrak{h} \ot A$ be a
Leibniz algebra homomorphism. We have to prove that there exists a unique algebra homomorphism $\theta : {\mathcal A}
(\mathfrak{h}, \, \mathfrak{g}) \to A$ such that the following diagram is commutative:
\begin{eqnarray} \eqlabel{diagrama10}
\xymatrix {& \mathfrak{g} \ar[r]^-{\eta_{\mathfrak{g}} } \ar[dr]_{f}
& {  \mathfrak{h} \ot {\mathcal A} (\mathfrak{h},
\, \mathfrak{g})} \ar[d]^{ {\rm Id}_{\mathfrak{h}} \ot \theta }\\
& {} & {\mathfrak{h} \ot A}} \qquad i.e. \,\,\, f = \bigl( {\rm Id}_{\mathfrak{h}} \ot \theta
\bigl) \circ \, \eta_{\mathfrak{g}}.
\end{eqnarray}
Let $\{d_{si} \, | \, s = 1,
\cdots, n, i\in I \}$ be a family of elements of $A$ such that for
any $i\in I$ we have:
\begin{equation}\eqlabel{constfmor}
f( f_i) = \sum_{s=1}^n \, e_s \ot d_{si}
\end{equation}
A straightforward computation shows that for all $i$, $j\in I$ we have:
$$
f \bigl( \left[f_i, \, f_j \right]_{\mathfrak{g}} \bigl) =
\sum_{a=1}^n \, e_a \ot \bigl ( \sum_{u \in B_{ij}} \, \beta_{i,
j}^u \, d_{au} \bigl) \,\, {\rm and} \,\,\left[ f(f_i), \, f(f_j)
\right]_{\mathfrak{h} \ot A} = \sum_{a=1}^n \, e_a \ot \bigl (
\sum_{s, t = 1}^n  \, \tau_{s, t}^a \, d_{si} d_{tj} \bigl)
$$
Since $f: \mathfrak{g} \to \mathfrak{h} \ot A$ is a Leibniz
algebra homomorphism, it follows that the family of elements $\{d_{si} \, |
\, s = 1, \cdots, n, i\in I \}$ need to fullfil the following relations in $A$:
\begin{equation}\eqlabel{deurile}
\sum_{u \in B_{ij}} \, \beta_{i, j}^u \, d_{au} = \sum_{s, t =
1}^n  \, \tau_{s, t}^a \, d_{si} d_{tj}, \quad {\rm for}\,\, {\rm all}\,\, i,\, j\in I\,\, {\rm and}\,\, a = 1, \cdots, n.
\end{equation}
The universal property of the polynomial algebra yields a unique algebra homomorphism $v : k [X_{si} \,
| s = 1, \cdots, n, \, i\in I] \to A$ such that $v (X_{si}) =
d_{si}$, for all $s = 1, \cdots, n$ and $i\in I$. It can be easily seen that
${\rm Ker} (v) \supseteq J$, where $J$ is the ideal generated by
all polynomials listed in \equref{poluniv}. Indeed, for any $i$, $j\in I$
and $a = 1, \cdots, n$ we have
\begin{eqnarray*}
v \bigl( P_{(a, i, j)} ^{(\mathfrak{h}, \, \mathfrak{g})}
 \bigl) = v \bigl(  \sum_{u \in
B_{i, j}} \, \beta_{i, j}^u \, X_{au} - \sum_{s, t = 1}^n \,
\tau_{s, t}^a \, X_{si} X_{tj} \bigl) = \sum_{u \in B_{i, j}} \, \beta_{i, j}^u \, d_{au} - \sum_{s,
t = 1}^n \, \tau_{s, t}^a \, d_{si} d_{tj}
\stackrel{\equref{deurile}} = 0.
\end{eqnarray*}
Thus, there exists a unique algebra homomorphism
$\theta : {\mathcal A} (\mathfrak{h}, \, \mathfrak{g}) \to A$ such
that $\theta (x_{si}) = d_{si}$, for all $s = 1, \cdots, n$ and
$i\in I$. Furthermore, for any $i\in I$ we have:
\begin{eqnarray*}
\bigl( {\rm Id}_{\mathfrak{h}} \ot \theta \bigl) \circ \,
\eta_{\mathfrak{g}} (f_i) = \bigl( {\rm Id}_{\mathfrak{h}} \ot
\theta \bigl) \bigl( \sum_{s=1}^n \, e_s \ot x_{si} \bigl) = \sum_{s=1}^n \, e_s \ot d_{si} \stackrel{\equref{constfmor}}
= f (f_i).
\end{eqnarray*}
Therefore, we have $\bigl( {\rm Id}_{\mathfrak{h}} \ot \theta \bigl) \circ \,
\eta_{\mathfrak{g}} = f$ as desired. Next we show that $\theta$ is the unique morphism with this property. Let $\tilde{\theta} : {\mathcal A}
(\mathfrak{h}, \, \mathfrak{g}) \to A$ be another algebra homomorphism such that $\bigl( {\rm Id}_{\mathfrak{h}} \ot
\tilde{\theta} \bigl) \circ \, \eta_{\mathfrak{g}} (f_i) = f
(f_i)$, for all $i\in I$. Then,  $\sum_{s=1}^n \, e_s \ot
\tilde{\theta} (x_{si}) = \sum_{s=1}^n \, e_s \ot d_{si}$, and
hence $\tilde{\theta} (x_{si}) = d_{si} = \theta (x_{si})$, for
all $s= 1, \cdots, n$ and $i\in I$. Since $\{x_{si} \, | s= 1, \cdots, n, i \in I \, \}$ is a
system of generators for the algebra ${\mathcal A} (\mathfrak{h},
\, \mathfrak{g})$ we obtain $\tilde{\theta} =
\theta$. All in all, we have proved that the map
$\gamma_{\mathfrak{g}, \, A}$ given by \equref{adjp} is bijective.

Next we show that assigning to each Leibniz algebra $\mathfrak{g}$
the commutative algebra ${\mathcal A} (\mathfrak{h}, \,
\mathfrak{g})$ defines a functor  ${\mathcal A} (\mathfrak{h}, \,
- ) : {\rm Lbz}_k \to {\rm ComAlg}_k$. First, let $\alpha :
\mathfrak{g}_1 \to \mathfrak{g}_2$ be a Leibniz
algebra homomorphism. Using the bijectivity of the map defined by
\equref{adjp} for the Leibniz algebra homomorphism $f :=
\eta_{\mathfrak{g}_2} \circ \alpha$, yields a unique algebra homomorphism $\theta : {\mathcal A} (\mathfrak{h}, \,
\mathfrak{g}_1) \to {\mathcal A} (\mathfrak{h}, \,
\mathfrak{g}_2)$ such that the following diagram is commutative:
\begin{eqnarray} \eqlabel{diagramapag4}
\xymatrix {& \mathfrak{g}_1 \ar[rr]^-{\eta_{\mathfrak{g}_1} }
\ar[d]_{\alpha} & {} & {  \mathfrak{h} \ot {\mathcal A} (\mathfrak{h},
\, \mathfrak{g}_1)} \ar[d]^{ {\rm Id}_{\mathfrak{h}} \ot \theta }\\
& \mathfrak{g}_2 \ar[rr]^-{\eta_{\mathfrak{g}_2}} & {} & {\mathfrak{h}
\ot {\mathcal A} (\mathfrak{h}, \, \mathfrak{g}_2) } } \qquad i.e. \,\,\, \bigl( {\rm Id}_{\mathfrak{h}} \ot \theta
\bigl) \circ \, \eta_{\mathfrak{g}_1} = \eta_{\mathfrak{g}_2}
\circ \alpha
\end{eqnarray}
We denote this unique morphism $\theta$ by  ${\mathcal A}
(\mathfrak{h}, \, \alpha )$ and the functor ${\mathcal A}
(\mathfrak{h}, \, - )$ is now fully defined. Furthermore, the
commutativity of the diagram \equref{diagramapag4} shows the
naturality of $\gamma_{\mathfrak{g}, \, A}$ in $\mathfrak{g}$. It
can now be easily checked that ${\mathcal A} (\mathfrak{h}, \, -
)$ is indeed a functor and that $\gamma_{\mathfrak{g}, \, A}$ is
also natural in $A$. To conclude, the functor ${\mathcal A}
(\mathfrak{h}, \, - )$ is a left adjoint of the current Leibniz
algebra functor $ \mathfrak{h} \ot -$.

Conversely, assume that the functor $\mathfrak{h}\ot - : {\rm
ComAlg}_k \to {\rm Lbz}_k$ has a left adjoint. In particular,
$\mathfrak{h}\ot - $ preserves arbitrary products. Now recall that
in both categories ${\rm ComAlg}_k$ and ${\rm Lbz}_k$ products are
constructed as simply the products of the underlying vector
spaces. Imposing the condition that $\mathfrak{h}\ot - $ preserves
the product of a countable number of copies of the base field $k$
will easily lead to the finite dimensionality of $\mathfrak{h}$.

Assume, now that the functor $\mathfrak{h}\ot - : {\rm ComAlg}_k
\to {\rm Lbz}_k$ has a right adjoint. This implies that $\mathfrak{h}\ot - $ preserves coproducts.
Now, since in the category ${\rm ComAlg}_k$ of commutative algebras the coproduct of two
commutative algebras is given by their tensor product, it follows that for any commutative algebras $A$
and $B$ there exists an isomorphism of Leibniz algebras
$\mathfrak{h} \ot (A \ot B) \cong (\mathfrak{h} \ot A) \, \sqcup
(\mathfrak{h} \ot B)$, where we denote by $\sqcup$ the coproduct of two current Leibniz algebras. In particular, for
$A = B := k$, we obtain that $\mathfrak{h}\cong  \mathfrak{h} \sqcup \mathfrak{h}$ and the corresponding morphisms $\mathfrak{h} \to \mathfrak{h} \sqcup \mathfrak{h}$ are just the identity maps. Therefore, for every Leibniz algebra $\mathfrak{g}$ there exists a unique Leibniz algebra homomorphism $\mathfrak{g} \to \mathfrak{h}$. Now by taking $\mathfrak{g} = \mathfrak{h} \times \mathfrak{h}$ and the embeddings $\mathfrak{h} \hookrightarrow \mathfrak{h} \times \mathfrak{h}$ to different components we reach a contradiction if $\mathfrak{h} \neq 0$.
\end{proof}

\begin{remark} \relabel{remar1}
\thref{adjunctie} remains valid in the special case of Lie
algebras: if $\mathfrak{h}$ is a finite dimensional Lie algebra,
the current Lie algebra functor $\mathfrak{h}\ot - : {\rm
ComAlg}_k \to {\rm Lie}_k$ has a left adjoint ${\mathcal A}
(\mathfrak{h}, \, -)$ which is constructed as in the proof of
\thref{adjunctie}. We point out, however, that the polynomials
defined in \equref{poluniv} take a rather simplified form. The
skew symmetry fulfilled by the bracket of a Lie algebra imposes the
following restrictions on the structure constants:
 $$\tau_{i,i}^s = 0\,\, {\rm and}\,\, \tau_{i,j}^s = - \tau_{j,i}^s\,\, {\rm for}\,\, {\rm all}\,\, i, j, s = 1,\cdots,
n.$$
\end{remark}

The commutative algebra ${\mathcal A} (\mathfrak{h}, \,
\mathfrak{g})$ constructed in the proof of \thref{adjunctie}
provides an important tool for studying Lie/Leibniz algebras as it
captures most of the essential information on the two Lie/Leibniz
algebras. Indeed, note for instance that the characters of this
algebra (i.e. the algebra homomorphisms ${\mathcal A}
(\mathfrak{h}, \, \mathfrak{g}) \to k$) parameterize the set of
all Leibniz algebra homomorphisms between the two algebras. This follows
as an easy consequence of the bijection described in \equref{adjp} by taking $A
:= k$:

\begin{corollary}\colabel{morlbz}
Let $\mathfrak{g}$ and $\mathfrak{h}$ be two Leibniz algebras such
that $\mathfrak{h}$ is finite dimensional. Then the following map is bijective:
\begin{equation}\eqlabel{adjpk}
\gamma \, : {\rm Hom}_{\rm Alg_k} \, ( {\mathcal A} (\mathfrak{h},
\, \mathfrak{g}), \, k) \to {\rm Hom}_{\rm Lbz_k} \,
(\mathfrak{g}, \, \mathfrak{h}), \quad \gamma (\theta) := \bigl(
{\rm Id}_{\mathfrak{h}} \ot \theta \bigl) \circ
\eta_{\mathfrak{g}}.
\end{equation}
\end{corollary}

In particular, by applying \coref{morlbz} for $\mathfrak{h} : =
\mathfrak{gl} (m, k)$ and an arbitrary Lie algebra $\mathfrak{g}$
we obtain:

\begin{corollary}\colabel{replie}
Let $\mathfrak{g}$ be a Lie algebra and $m$ a positive integer.
Then there exists a bijective correspondence between the space of
all $m$-dimensional representations of $\mathfrak{g}$ and the
space of all algebra homomorphisms ${\mathcal A}
(\mathfrak{gl} (m, k), \, \mathfrak{g}) \to k$.
\end{corollary}

\begin{examples} \exlabel{exgen}
1. If $\mathfrak{h}$ and $\mathfrak{g}$ are abelian Leibniz
algebras then ${\mathcal A} (\mathfrak{h}, \, \mathfrak{g}) \cong
k [X_{si} \, | s = 1, \cdots, n, \, i\in I]$, where $n = {\rm
dim}_k (\mathfrak{h})$ and $|I| = {\rm dim}_k (\mathfrak{g})$.

2. Let $\mathfrak{h}$ be an $n$-dimensional Leibniz algebra with
structure constants $\{\tau_{i, j}^s \, | \, i, j, s = 1, \cdots,
n \}$. Then ${\mathcal A} (\mathfrak{h}, \, k) \cong k[X_1,
\cdots, X_n]/J$, where $J$ is the ideal generated by the
polynomials $\sum_{s, t = 1}^n \, \tau_{s, t}^a \, X_{s} X_{t}$,
for all $a = 1, \cdots, n$.

3. Let $\mathfrak{g}$ be a Leibniz algebra. Then ${\mathcal A} (k,
\, \mathfrak{g}) \cong S (\mathfrak{g}/\mathfrak{g}')$, the
symmetric algebra of $\mathfrak{g}/\mathfrak{g}'$, where
$\mathfrak{g}'$ is the derived subalgebra of $\mathfrak{g}$. In
particular, if $\mathfrak{g}$ is perfect (that is $\mathfrak{g}' =
\mathfrak{g}$), then ${\mathcal A} (k, \, \mathfrak{g}) \cong k$.

Indeed, the functor ${\mathcal A} (k, \, -)$ is a left adjoint for
the tensor functor $k \ot - : {\rm ComAlg}_k \to {\rm Lbz}_k$;
since the tensor product is also taken over $k$ this functor is
isomorphic to the functor $(-)_0 : {\rm ComAlg}_k \to {\rm
Lbz}_k$, which sends any commutative algebra $A$ to the abelian Leibniz algebra $A_0 := A$. We shall prove that the functor
$\mathfrak{g} \mapsto S (\mathfrak{g}/\mathfrak{g}')$ is a left
adjoint of $(-)_0$. The uniqueness of adjoint functors
\cite{mlane} will then lead to the desired algebra isomorphism
${\mathcal A} (k, \, \mathfrak{g}) \cong S
(\mathfrak{g}/\mathfrak{g}')$.

Let $\mathfrak{g}$ be a Leibniz algebra and define
$\overline{\eta_{\mathfrak{g}}} : \mathfrak{g} \to S
(\mathfrak{g}/\mathfrak{g}')$ as the composition
$\overline{\eta_{\mathfrak{g}}} := i \circ \pi$, where $\pi:
\mathfrak{g} \to \mathfrak{g}/\mathfrak{g}'$ is the usual
projection and $i : \mathfrak{g}/\mathfrak{g}' \hookrightarrow S
(\mathfrak{g}/\mathfrak{g}')$ is the canonical inclusion of the
vector space $\mathfrak{g}/\mathfrak{g}'$ in its symmetric
algebra. We shall prove now that the following map is bijective for any commutative algebra $A$ and any Leibniz
algebra $\mathfrak{g}$:
\begin{equation}\eqlabel{adjdoi}
\overline{\gamma_{\mathfrak{g}, \, A}} \, : {\rm Hom}_{\rm Alg_k}
\, ( S (\mathfrak{g}/\mathfrak{g}'), \, A) \to {\rm Hom}_{\rm
Lbz_k} \, (\mathfrak{g}, \, A_0), \quad
\overline{\gamma_{\mathfrak{g}, \, A}} \, (\theta) := \theta \circ
\overline{\eta_{\mathfrak{g}}}
\end{equation}
This shows that the functor $\mathfrak{g}
\mapsto S (\mathfrak{g}/\mathfrak{g}')$ is a left adjoint of
$(-)_0$. Indeed, let $f: \mathfrak{g} \to A_0$ be a Leibniz
algebra homomorphism, i.e. $f$ is a $k$-linear map such that $f (\left[x, \, y
\right]) = 0$, for all $x$, $y\in \mathfrak{g}$. That is ${\rm
Ker(f)}$ contains $\mathfrak{g}'$, the derived algebra of
$\mathfrak{g}$. Thus, there exists a unique $k$-linear map
$\overline{f} : \mathfrak{g}/\mathfrak{g}' \to A$ such that $
\overline{f} \circ \pi = f$. Now, using the universal property of the
symmetric algebra we obtain that there exists a unique algebra homomorphism $\theta : S (\mathfrak{g}/\mathfrak{g}') \to
A$ such that $\theta \circ i = \overline{f}$, and hence
$\overline{\gamma_{\mathfrak{g}, \, A}} \, (\theta) = f$. Therefore, the map $\overline{\gamma_{\mathfrak{g}, \, A}}$ is bijective and
the proof is now finished.
\end{examples}

\begin{definition} \delabel{alguniv}
Let $\mathfrak{g}$ and $\mathfrak{h}$ be Leibniz algebras with
$\mathfrak{h}$ finite dimensional. Then the commutative algebra
${\mathcal A} (\mathfrak{h}, \, \mathfrak{g})$ is called the
\emph{universal algebra} of $\mathfrak{h}$ and  $\mathfrak{g}$.
When $\mathfrak{h} =  \mathfrak{g}$ we denote the universal
algebra of $\mathfrak{h}$ simply by ${\mathcal A} (\mathfrak{h})$.
\end{definition}

If $\{\tau_{i, j}^s \, | \, i, j, s = 1, \cdots, n \}$ are the
structure constants of $\mathfrak{h}$, where $n$ is the dimension
of $\mathfrak{h}$, then the polynomials defined for any $a$, $i$,
$j = 1, \cdots, n$ by:
\begin{equation}\eqlabel{poluniv2}
P_{(a, i, j)} ^{(\mathfrak{h})} := \sum_{u = 1}^n \, \tau_{i, j}^u
\, X_{au} - \sum_{s, t = 1}^n \, \tau_{s, t}^a \, X_{si} X_{tj} \,
\in k[X_{ij} \, | \, i, j = 1, \cdots, n]
\end{equation}
are called the \emph{universal polynomials} of $\mathfrak{h}$. It
follows from the proof of \thref{adjunctie} that ${\mathcal A}
(\mathfrak{h}) = k[X_{ij} \, | \, i, j = 1, \cdots, n]/J$, where
$J$ is the ideal generated by the universal polynomials $P_{(a, i,
j)} ^{(\mathfrak{h})}$, for all $a$, $i$, $j = 1, \cdots, n$.
Moreover, if $\{e_1, \cdots, e_n\}$ is a basis in $\mathfrak{h}$
then the canonical map
\begin{equation}\eqlabel{unitadj2}
\eta_{\mathfrak{h}} : \mathfrak{h} \to \mathfrak{h} \ot {\mathcal
A} (\mathfrak{h}), \quad \eta_{\mathfrak{h}} (e_i) := \sum_{s=1}^n
\, e_s \ot x_{si}
\end{equation}
for all $i = 1, \cdots, n$ is a Leibniz algebra homomorphism. The
commutative algebra ${\mathcal A} (\mathfrak{h})$ and the family
of polynomials $P_{(a, i, j)} ^{(\mathfrak{h})}$ are purely
algebraic objects that capture the entire information of the
Leibniz algebra $\mathfrak{h}$. Moreover, the universal algebra
${\mathcal A} (\mathfrak{h})$ satisfies the following universal
property:

\begin{corollary}\colabel{initialobj}
Let $\mathfrak{h}$ be a finite dimensional Leibniz algebra. Then
for any commutative algebra $A$ and any Leibniz
algebra homomorphism $f : \mathfrak{h} \to \mathfrak{h} \ot A$, there exists a
unique algebra homomorphism $\theta: {\mathcal A}
(\mathfrak{h}) \to A$ such that $f = ({\rm Id}_{\mathfrak{h}} \ot
\theta) \circ \eta_{\mathfrak{h}}$, i.e. the following diagram is commutative:
\begin{eqnarray} \eqlabel{univerah}
\xymatrix {& \mathfrak{h} \ar[rr]^-{\eta_{\mathfrak{h}} } \ar[rrd]_{ f
} & {} & {\mathfrak{h} \ot {\mathcal A} (\mathfrak{h} )} \ar[d]^{ {\rm Id}_{\mathfrak{h}} \ot \theta }\\
& {}  & {} &
{\mathfrak{h} \ot A} }
\end{eqnarray}
\end{corollary}

\begin{proof}
Follows straightforward from the bijection given in \equref{adjp} for
$\mathfrak{g}:= \mathfrak{h}$.
\end{proof}

\begin{remark} \relabel{remar2}
If $\mathfrak{h}$ is a Lie algebra of dimension $n$, then the
structure constants are subject to the following relations $\tau_{i,i}^s = 0$ and
$\tau_{i,j}^s = - \tau_{j,i}^s$, for all $i$, $j$, $s = 1,\cdots,
n$. Consequently, we can easily see that the universal polynomials
of $\mathfrak{h}$ fulfill the following conditions:
\begin{equation}\eqlabel{liepol}
P_{(a, i, i)} ^{(\mathfrak{h})} = 0 \quad \quad {\rm and }\quad
\quad P_{(a, i, j)} ^{(\mathfrak{h})} = - P_{(a, j, i)}
^{(\mathfrak{h})}
\end{equation}
for all $a$, $i$, $j = 1, \cdots n$, $i \neq j$. Thus, in the case
of Lie algebras the universal algebra ${\mathcal A}
(\mathfrak{h})$ takes a simplified form. We provide further
examples in the sequel.
\end{remark}

\begin{examples} \exlabel{unvlie}
1. Let $\mathfrak{h} := {\rm aff} (2, k)$ be the affine
$2$-dimensional Lie algebra with basis $\{e_1, e_2\}$ and
bracket given by $\left[e_1, \, e_2 \right] = e_1$. Then, we have:
\begin{eqnarray*}
{\mathcal A} ({\rm aff} (2, k)) &\cong& \, k [X_{11}, X_{12},
X_{21}, X_{22}]/(X_{21}, \, X_{11} - X_{12}X_{22} + X_{12}X_{21}) \\
& \cong& \, k[X, Y, Z]/(X - YZ)
\end{eqnarray*}
Indeed, the non-zero structure constants of $\mathfrak{h}$
are $\tau_{1,2}^1 = 1 = - \tau_{2,1}^1$. Using \equref{liepol} from the previous remark
the only non-zero universal polynomials of the Lie algebra ${\rm aff} (2, k)$ are $P_{(1, 1, 2)} = X_{11} - X_{12}X_{22}
+ X_{12}X_{21}$, $P_{(2, 1, 2)} = X_{21}$, $-P_{(1, 1, 2)}$ and
$-P_{(2, 1, 2)}$. The conclusion now follows.

2. Let $\mathfrak{h} :=  \mathfrak{sl}(2, k)$ be the Lie algebra
with basis $\{e_1, e_2, e_3\}$ and bracket $\left[e_1, \, e_2 \right] = e_3$, $\left[e_3, \, e_2 \right] = -2
e_2$, $\left[e_3, \, e_1 \right] = 2e_1$. A routinely computation
proves that ${\mathcal A} (\mathfrak{sl}(2, k)) \cong k[X_{ij} \,
| \, i, j = 1, 2, 3]/J$, where $J$ is the ideal generated by the
following nine universal polynomials of $\mathfrak{sl}(2, k)$:
\begin{eqnarray*}
&& \hspace*{-10mm} X_{13} - 2 X_{12}X_{31} + 2X_{11}X_{32}, \,\,\, 2X_{11} -
2X_{11}X_{33} + 2 X_{13}X_{31},\,\,\,  2X_{12} - 2X_{13}X_{32} +
2X_{12}X_{33}\\
&& \hspace*{-10mm} X_{23} - 2 X_{21}X_{32} + 2X_{22}X_{31}, \,\,\, 2X_{21} -
2X_{23}X_{31} + 2 X_{21}X_{33},\,\,\, 2X_{22} - 2X_{22}X_{33} +
2X_{23}X_{32}\\
&& \hspace*{-10mm} X_{33} - X_{11}X_{22} + X_{12}X_{21}, \,\,\, 2X_{31} - X_{21}X_{13}
+  X_{11}X_{23}, \,\,\, 2X_{32} - X_{12}X_{23} + X_{13}X_{22}.
\end{eqnarray*}
\end{examples}

We recall that the polynomial algebra $M(n) = k[X_{ij} \, | \, i,
j = 1, \cdots, n]$ is a bialgebra with comultiplication and counit
given by $\Delta (X_{ij}) = \sum_{s=1}^n \, X_{is} \ot X_{sj}$ and
$\varepsilon (X_{ij}) = \delta_{i, j}$, for any $i$, $j=1, \cdots,
n$. We will prove now that the universal algebra ${\mathcal A}
(\mathfrak{h})$ is also a bialgebra.

\begin{proposition} \prlabel{bialgebra}
Let $\mathfrak{h}$ be a Leibniz algebra of dimension $n$. Then
there exists a unique bialgebra structure on ${\mathcal A}
(\mathfrak{h})$ such that the Leibniz algebra homomorphism
$\eta_{\mathfrak{h}} : \mathfrak{h} \to \mathfrak{h} \ot {\mathcal
A} (\mathfrak{h})$ becomes a right ${\mathcal A}
(\mathfrak{h})$-comodule structure on $\mathfrak{h}$. More
precisely, the comultiplication and the counit on ${\mathcal A}
(\mathfrak{h})$ are given for any $i$, $j=1, \cdots, n$ by
\begin{equation} \eqlabel{deltaeps}
\Delta (x_{ij}) = \sum_{s=1}^n \, x_{is} \ot x_{sj} \quad {\rm
and} \quad  \varepsilon (x_{ij}) = \delta_{i, j}
\end{equation}
Furthermore, the usual projection $\pi \colon M(n) \to {\mathcal A} (\mathfrak{h})$ becomes a bialgebra homomorphism.
\end{proposition}

\begin{proof}
Consider the Leibniz algebra homomorphism $f : \mathfrak{h} \to
\mathfrak{h} \ot {\mathcal A} (\mathfrak{h}) \ot {\mathcal A}
(\mathfrak{h})$ defined by $f := (\eta_{\mathfrak{h}} \ot {\rm
Id}_{{\mathcal A} (\mathfrak{h})} ) \, \circ \,
\eta_{\mathfrak{h}}$. It follows from \coref{initialobj} that
there exists a unique algebra homomorphism $\Delta :
{\mathcal A} (\mathfrak{h}) \to {\mathcal A} (\mathfrak{h}) \ot
{\mathcal A} (\mathfrak{h})$ such that $({\rm Id}_{\mathfrak{h}}
\ot \Delta) \circ \eta_{\mathfrak{h}} = f$; that is, the following
diagram is commutative:
\begin{eqnarray} \eqlabel{delta}
\xymatrix {& \mathfrak{h} \ar[rr]^-{\eta_{\mathfrak{h}} } \ar[d]_{
\eta_{\mathfrak{h}} } & {} & {\mathfrak{h} \ot
{\mathcal A} (\mathfrak{h} )} \ar[d]^{ {\rm Id}_{\mathfrak{h}} \ot \Delta }\\
& \mathfrak{h} \ot {\mathcal A} (\mathfrak{h} )
\ar[rr]_-{\eta_{\mathfrak{h}}  \ot {\rm Id}_{{\mathcal A} (\mathfrak{h})}} & {} & {\mathfrak{h} \ot
{\mathcal A} (\mathfrak{h} ) \ot {\mathcal A} (\mathfrak{h} )} }
\end{eqnarray}
Now, if we evaluate the diagram \equref{delta} at
each $e_i$, for $i = 1, \cdots, n$ we obtain, taking into
account \equref{unitadj2}, the following:
\begin{eqnarray*}
&& \sum_{t=1}^n \, e_t \ot \Delta (x_{ti}) = (\eta_{\mathfrak{h}}
\ot {\rm Id}) (\sum_{s=1}^n \, e_s \ot x_{si}) = \sum_{s=1}^n (
\sum_{t=1}^n \, e_t \ot x_{ts}) \ot x_{si}\\
&& = \sum_{t=1}^n \, e_t \ot (\sum_{s=1}^n x_{ts} \ot x_{si} )
\end{eqnarray*}
and hence $\Delta (x_{ti}) = \sum_{s=1}^n \, x_{ts} \ot x_{si}$,
for all $t$, $i=1, \cdots, n$. Obviously, $\Delta$ given by this
formula on generators is coassociative. In a similar fashion,
applying once again \coref{initialobj}, we obtain that there
exists a unique algebra homomorphism $\varepsilon:
{\mathcal A} (\mathfrak{h}) \to k$ such that the following diagram
is commutative:
\begin{eqnarray} \eqlabel{epsilo}
\xymatrix {& \mathfrak{h} \ar[rr]^-{\eta_{\mathfrak{h}} } \ar[drr]_{
{\rm can}
} & {} & {\mathfrak{h} \ot {\mathcal A} (\mathfrak{h} )}
\ar[d]^{ {\rm Id}_{\mathfrak{h}} \ot \varepsilon }\\
& {} & {} & {\mathfrak{h} \ot k} }
\end{eqnarray}
where ${\rm can} : \mathfrak{h} \to \mathfrak{h}
\ot k$ is the canonical isomorphism, ${\rm can} (x) = x \ot 1$,
for all $x\in \mathfrak{h}$. If we evaluate this diagram at each
$e_t$, for $t = 1, \cdots, n$, we obtain $\varepsilon
(x_{ij}) = \delta_{i, j}$, for all $i$, $j=1, \cdots, n$. It can be easily checked that
$\varepsilon$ is a counit for $\Delta$, thus ${\mathcal A}
(\mathfrak{h})$ is a bialgebra. Furthermore, the commutativity of the above two diagrams
imply that the
canonical map $\eta_{\mathfrak{h}} : \mathfrak{h} \to \mathfrak{h}
\ot {\mathcal A} (\mathfrak{h})$ defines a right ${\mathcal
A} (\mathfrak{h})$-comodule structure on $\mathfrak{h}$.
\end{proof}

We call the pair $({\mathcal A}(\mathfrak{h}), \,
\eta_{\mathfrak{h}} )$, with the coalgebra structure defined in
\prref{bialgebra}, the {\it universal coacting bialgebra of the
Leibniz algebra $\mathfrak{h}$}. It fullfils the following
universal property which extends \coref{initialobj}:

\begin{theorem}\thlabel{univbialg}
Let $\mathfrak{h}$ be a Leibniz algebra of dimension $n$. Then,
for any commutative bialgebra $B$ and any Leibniz algebra
homomorphism $f \colon \mathfrak{h} \to \mathfrak{h} \otimes B$
which makes $\mathfrak{h}$ into a right $B$-comodule there exists
a unique bialgebra homomorphism $\theta \colon {\mathcal A}
(\mathfrak{h}) \to B$ such that the following diagram is
commutative:
\begin{eqnarray} \eqlabel{univbialg}
\xymatrix {& \mathfrak{h} \ar[r]^-{\eta_{\mathfrak{h}}} \ar[dr]_-{f
} & {\mathfrak{h} \ot {\mathcal A} (\mathfrak{h} )} \ar[d]^{ {\rm Id}_{\mathfrak{h}} \ot \theta }\\
& {}  &
{\mathfrak{h} \ot B} }
\end{eqnarray}
\end{theorem}

\begin{proof}
As ${\mathcal A} (\mathfrak{h})$ is the universal algebra of
$\mathfrak{h}$, there exists a unique algebra homomorphism $\theta
\colon {\mathcal A} (\mathfrak{h}) \to B$ such that diagram~\equref{univbialg} commutes. The proof will be finished
once we show that $\theta$ is a coalgebra homomorphism as well.
This follows by using again the universal property of ${\mathcal
A} (\mathfrak{h})$. Indeed, we obtain a unique algebra
homomorphism $\psi \colon {\mathcal A} (\mathfrak{h}) \to B
\otimes B$ such that the following diagram is commutative:
\begin{equation}\eqlabel{101}
\xymatrix{ \mathfrak{h}\ar[r]^-{\eta_{\mathfrak{h}} }\ar[rdd]_{\bigl({\rm Id}_{\mathfrak{h}}\otimes\,
\Delta_{B} \circ \theta\bigl)\circ \eta_{\mathfrak{h}} }  & {\mathfrak{h} \otimes {\mathcal A} (\mathfrak{h})}\ar[dd]^{{\rm Id}_{\mathfrak{h}} \otimes \psi}  \\
{} & {} \\
{} & {\mathfrak{h} \otimes B \otimes B}
}
\end{equation}
The proof will be finished once we show that $(\theta \otimes
\theta) \circ \Delta$ makes diagram~\equref{101} commutative.
Indeed, as  $f \colon
\mathfrak{h} \to \mathfrak{h} \otimes B$ is a right $B$-comodule
structure, we have:
\begin{eqnarray*}
\bigl({\rm Id}_{\mathfrak{h}} \otimes\, (\theta \otimes \theta) \circ \Delta\bigl)\circ \,\eta_{\mathfrak{h}} &=& \bigl({\rm Id}_{\mathfrak{h}} \otimes \theta \otimes \theta \bigl)\circ \underline{\bigl({\rm Id}_{\mathfrak{h}} \otimes \Delta\bigl)\circ \, \eta_{\mathfrak{h}}}\\
&\stackrel{\equref{delta}} {=}& \bigl({\rm Id}_{\mathfrak{h}} \otimes \theta \otimes \theta \bigl)\circ (\eta_{\mathfrak{h}} \otimes {\rm Id}_{{\mathcal A} (\mathfrak{h})})\circ \eta_{\mathfrak{h}}\\
&=& \bigl(\underline{({\rm Id}_{\mathfrak{h}} \otimes \theta) \circ \eta_{\mathfrak{h}}}\ \otimes \theta\bigl)\circ \, \eta_{\mathfrak{h}}\\\
&\stackrel{\equref{univbialg}} {=}& \bigl(f \otimes \theta\bigl)\circ \, \eta_{\mathfrak{h}}\\
&=& (f \otimes {\rm Id}_{B})\circ \underline{({\rm Id}_{\mathfrak{h}} \otimes \theta) \circ \, \eta_{\mathfrak{h}}}\\
&\stackrel{\equref{univbialg}} {=}& \underline{(f \otimes  {\rm Id}_{B})\circ f}\\
&=& ({\rm Id}_{\mathfrak{h}} \otimes \Delta_{B}) \circ \underline{f}\\
&\stackrel{\equref{univbialg}} {=}& ({\rm Id}_{\mathfrak{h}} \otimes \Delta_{B}) \circ ({\rm Id}_{\mathfrak{h}} \otimes \theta) \circ \eta_{\mathfrak{h}}\\
&=& ({\rm Id}_{\mathfrak{h}} \otimes \Delta_{B} \circ \theta) \circ \eta_{\mathfrak{h}}
\end{eqnarray*}
as desired. Similarly, one can show that  $\varepsilon_B \, \circ \, \theta = \varepsilon$ and the proof is now finished.
\end{proof}

In what follows we construct for any finite dimensional Leibniz
algebra $\mathfrak{h}$ a universal commutative Hopf algebra
${\mathcal H} (\mathfrak{h})$ together with a Leibniz algebra
homomorphism $\lambda_{\mathfrak{h}} \colon \mathfrak{h} \to
\mathfrak{h} \ot {\mathcal H} (\mathfrak{h})$ which makes
$\mathfrak{h}$ into a right ${\mathcal H}
(\mathfrak{h})$-comodule. This is achieved by using the free
commutative Hopf algebra generated by a commutative bialgebra
introduced in \cite[Chapter IV]{T}. Recall that assigning to a
commutative bialgebra the free commutative Hopf algebra defines a
functor $L \colon {\rm ComBiAlg}_k \to {\rm ComHopf}_k$ which is a
left adjoint to the forgetful functor ${\rm ComHopf}_k \to {\rm
ComBiAlg}_k$ (\cite[Theorem 65, (2)]{T}). Throughout, we denote by
$\mu \colon 1_{{\rm ComBiAlg}_k} \to UL$ the unit of the
adjunction $L \dashv U$.

\begin{definition}
Let $\mathfrak{h}$ be a finite dimensional Leibniz algebra. The
pair $\bigl({\mathcal H} (\mathfrak{h}) := L({\mathcal A}
(\mathfrak{h})), \, \lambda_{\mathfrak{h}} := ({\rm
Id}_{\mathfrak{h}} \ot \mu_{{\mathcal A} (\mathfrak{h})}) \, \circ
\,  \eta_{\mathfrak{h}}\bigl)$ is called the {\it universal
coacting Hopf algebra of $\mathfrak{h}$}.
\end{definition}

The pair $\bigl( {\mathcal H} (\mathfrak{h}), \,
\lambda_{\mathfrak{h}} \bigl)$ fulfills the following universal
property which shows that it is the initial object in the category
of all commutative Hopf algebras that coact on $\mathfrak{h}$.

\begin{theorem} \thlabel{univhopf}
Let $\mathfrak{h}$ be a finite dimensional Leibniz algebra. Then,
for any commutative Hopf algebra $H$ and any Leibniz algebra
homomorphism $f \colon \mathfrak{h} \to \mathfrak{h} \otimes H$
which makes $\mathfrak{h}$ into a right $H$-comodule there exists
a unique Hopf algebra homomorphism $g \colon {\mathcal H}
(\mathfrak{h}) \to H$ for which the following diagram is
commutative:
\begin{eqnarray} \eqlabel{univHopfalg}
\xymatrix {& \mathfrak{h} \ar[rr]^-{\lambda_{\mathfrak{h}} } \ar[drr]_{ f
} & {} & {\mathfrak{h} \ot {\mathcal H} (\mathfrak{h} )} \ar[d]^{ {\rm Id}_{\mathfrak{h}} \ot g }\\
& {}  & {} &
{\mathfrak{h} \ot H} }
\end{eqnarray}
\end{theorem}

\begin{proof}
Let $H$ be a commutative Hopf algebra together with a Leibniz
algebra homomorphism $f \colon \mathfrak{h} \to \mathfrak{h}
\otimes H$ which makes $\mathfrak{h}$ into right a $H$-comodule.
Using \thref{univbialg} we obtain a unique bialgebra homomorphism
$\theta: {\mathcal A} (\mathfrak{h})  \to H$ which makes the
following diagram commutative:
\begin{eqnarray}\label{final1}
\xymatrix {& \mathfrak{h} \ar[rr]^-{\eta_{\mathfrak{h}} } \ar[drr]_{ f
} & {} & {\mathfrak{h} \ot {\mathcal A} (\mathfrak{h} )} \ar[d]^{ {\rm Id}_{\mathfrak{h}} \ot \theta }\\
& {}  & {} &
{\mathfrak{h} \ot H} }\qquad i.e.\,\,\,({\rm Id}_{\mathfrak{h}} \ot \theta ) \circ \eta_{\mathfrak{h}} = f.
\end{eqnarray}
Now the adjunction $L  \dashv U$ yields a unique Hopf algebra
homomorphism $g \colon L({\mathcal A} (\mathfrak{h})) \to H$ such
that the following diagram commutes:
\begin{eqnarray}\label{final2}
\xymatrix{ {\mathcal A} (\mathfrak{h})\ar[rr]^-{\mu_{{\mathcal A}
(\mathfrak{h})}}\ar[rrd]_{\theta}  & {} & {L({\mathcal A} (\mathfrak{h}))}\ar[d]^{ g}  \\
{} & {} & {H}
} \qquad {\rm i.e.}\,\,\, g \circ \mu_{{\mathcal A} (\mathfrak{h})} = \theta.
\end{eqnarray}
We are now ready to show that $g \colon {\mathcal H}
(\mathfrak{h}) = L({\mathcal A} (\mathfrak{h})) \to H$ is the
unique Hopf algebra homomorphism which makes diagram
\equref{univHopfalg} commutative. Indeed, putting all the above
together yields:
\begin{eqnarray*}
({\rm Id}_{\mathfrak{h}} \otimes g) \circ ({\rm Id}_{\mathfrak{h}} \ot \mu_{{\mathcal A} (\mathfrak{h})}) \circ \eta_{\mathfrak{h}}
&=& ({\rm Id}_{\mathfrak{h}} \ot \underline{g \circ \mu_{{\mathcal A} (\mathfrak{h})}}) \circ \eta_{\mathfrak{h}} \\
&\stackrel{(\ref{final2})} {=}& \underline{({\rm Id}_{\mathfrak{h}} \ot \theta) \circ \eta_{\mathfrak{h}}} \\
&\stackrel{(\ref{final1})} {=}& f.
\end{eqnarray*}
Since $g$ is obviously the unique Hopf algebra homomorphism which makes the above diagram commutative, the proof is finished.
\end{proof}

\section{Applications: the automorphism group and the classification of gradings on Leibniz algebras}\selabel{sect3}

In this section we discuss three applications of our previous
results which highlight the importance of the newly introduced universal coacting
bialgebra (Hopf algebra) of a Leibniz algebra. The first one concerns the description of the automorphism group ${\rm Aut}_{{\rm Lbz}}
(\mathfrak{h})$ of a given Leibniz/Lie algebra $\mathfrak{h}$
which is a classical and notoriously difficult problem arising
form Hilbert's invariant theory. We start by recalling a few basic facts from the theory of Hopf algebras
\cite{Sw, radford} which will be useful in the sequel. For any bialgebra $H$ the set of group-like
elements, denoted by $G(H) := \{g\in H \, | \, \Delta (g) = g \ot
g \,\, {\rm and } \,\, \varepsilon(g) = 1 \}$ is a monoid with
respect to the multiplication of $H$. We denote by $H^{\rm o}$,
the finite dual bialgebra of $H$, i.e.:
$$
H^{\rm o} := \{ f \in H^* \,| \, f(I) = 0, \, {\rm for \, some \,
ideal} \,\, I \lhd H \,\, {\rm with} \,\, {\rm dim}_k (H/I) <
\infty \}
$$
It is well known (see for instance \cite[pag. 62]{radford}) that
$G(H^{\rm o}) = {\rm Hom}_{\rm Alg_k} (H, \,  k)$, the set of all
algebra homomorphisms $H\to k$. Now, we shall give the first
application of the universal bialgebra of a Leibniz algebra.

\begin{theorem} \thlabel{automorf}
Let $\mathfrak{h}$ be a finite dimensional Leibniz algebra with
basis $\{e_1, \cdots, e_n\}$ and consider $U\bigl (G\bigl( {\mathcal A}
(\mathfrak{h})^{\rm o} \bigl)\bigl)$ to be the group of all
invertible group-like elements of the finite dual ${\mathcal A}
(\mathfrak{h})^{\rm o}$. Then the map defined for any $\theta \in
U\bigl(G\bigl( {\mathcal A} (\mathfrak{h})^{\rm o} \bigl)\bigl)$
and $i = 1, \cdots, n$ by:
\begin{equation} \eqlabel{izomono}
\overline{\gamma} : U \bigl(G\bigl( {\mathcal A}
(\mathfrak{h})^{\rm o} \bigl) \bigl) \to {\rm Aut}_{{\rm Lbz}}
(\mathfrak{h}), \qquad \overline{\gamma} (\theta) (e_i) :=
\sum_{s=1}^n \, \theta(x_{si}) \, e_s
\end{equation}
is an isomorphism of groups.
\end{theorem}

\begin{proof}
By applying \coref{morlbz} for $\mathfrak{g}:= \mathfrak{h}$ it
follows that the map
$$
\gamma : {\rm Hom}_{\rm Alg_k} ({\mathcal A} (\mathfrak{h}) , \,
k) \to {\rm End}_{{\rm Lbz}} (\mathfrak{h}), \quad \gamma (\theta)
= \bigl( {\rm Id}_{\mathfrak{h}} \ot \theta \bigl) \circ
\eta_{\mathfrak{g}}
$$
is bijective. Based on formula \equref{unitadj2}, it can be easily
seen that $\gamma$ takes the form given in \equref{izomono}. As mentioned above we have $ {\rm Hom}_{\rm Alg_k} ({\mathcal A}
(\mathfrak{h}) , k) = G\bigl( {\mathcal A} (\mathfrak{h})^{\rm o}
\bigl)$. Therefore, since $\overline{\gamma}$ is the restriction
of $\gamma$ to the invertible elements of the two monoids, the
proof will be finished once we show that $\gamma$ is an
isomorphism of monoids. To this end, recall that the monoid structure on
${\rm End}_{{\rm Lbz}} (\mathfrak{h})$ is given by the usual
composition of endomorphisms of the Leibniz algebra
$\mathfrak{h}$, while $G\bigl( {\mathcal A} (\mathfrak{h})^{\rm o}
\bigl)$ is a monoid with respect to the convolution product, that
is:
\begin{equation}\eqlabel{convolut}
(\theta_1 \star \theta_2) (x_{sj}) = \sum_{t=1}^n \,
\theta_1(x_{st}) \theta_2(x_{tj})
\end{equation}
for all $\theta_1$, $\theta_2 \in G\bigl( {\mathcal A}
(\mathfrak{h})^{\rm o} \bigl)$ and $j$, $s = 1, \cdots, n$. Now,
for any $\theta_1$, $\theta_2 \in G\bigl( {\mathcal A}
(\mathfrak{h})^{\rm o} \bigl)$ and $j = 1, \cdots, n$ we have:
\begin{eqnarray*}
&& \bigl(\gamma(\theta_1) \circ \gamma(\theta_2) \bigl) (e_j) =
\gamma(\theta_1) \bigl( \sum_{t=1}^n \, \theta_2 (x_{tj}) e_t
\bigl) = \sum_{s, t = 1}^n \, \theta_1(x_{st}) \theta_2 (x_{tj})\,
e_s \\
&& = \sum_{s=1}^n \, \bigl( \sum_{t=1}^n \, \theta_1(x_{st})
\theta_2 (x_{tj}) \bigl) \, e_s = \sum_{s=1}^n \, (\theta_1 \star
\theta_2) (x_{sj}) \, e_s = \gamma (\theta_1 \star \theta_2) (e_j)
\end{eqnarray*}
thus, $\gamma (\theta_1 \star \theta_2) = \gamma(\theta_1) \circ
\gamma(\theta_2)$, and therefore $\gamma$ respects the
multiplication. We are left to show that $\gamma$ also preserves
the unit. Note that the unit $1$ of the monoid $G\bigl( {\mathcal
A} (\mathfrak{h})^{\rm o} \bigl)$ is the counit
$\varepsilon_{{\mathcal A} (\mathfrak{h})}$ of the bialgebra
${\mathcal A} (\mathfrak{h})$ and we obtain:
$$
\gamma(1) (e_i) = \gamma (\varepsilon_{{\mathcal A}
(\mathfrak{h})}) (e_i) = \sum_{s=1}^n \, \varepsilon_{{\mathcal A}
(\mathfrak{h})} (x_{si}) \, e_s = \sum_{s=1}^n \, \delta_{si} \,
e_s = e_i = {\rm Id}_{\mathfrak{h}} (e_i)
$$
Thus we have proved that $\gamma$ is an isomorphism of monoids and
the proof is finished.
\end{proof}

\begin{remark}
We point out that the construction of ${\mathcal A} (\mathfrak{h})$, as well as ${\mathcal A} (\mathfrak{h}, \mathfrak{g})$, and the description of the automorphism group of $\mathfrak{h}$ can be achieved for an arbitrary finite dimensional algebra $\mathfrak{h}$, not necessarily Lie or Leibniz. This avenue of investigation is considered in a a forthcoming paper of the authors. 
\end{remark}

The second application we consider is related
to the classical problem of classifying all $G$-gradings on
a given Leibniz/Lie algebras. Let $G$ be an abelian group and
$\mathfrak{h}$ a Leibniz algebra. Recall that a
\emph{$G$-grading} on $\mathfrak{h}$ is a vector space
decomposition $\mathfrak{h} = \oplus_{\sigma \in G} \,
\mathfrak{h}_{\sigma}$ such that $\left [\mathfrak{h}_{\sigma}, \,
\mathfrak{h}_{\tau} \right] \subseteq \mathfrak{h}_{\sigma \tau}$
for all $\sigma$, $\tau \in G$. For more detail on the problem of
classifying $G$-gradings on Lie algebras see \cite{eld} and the
references therein. In what follows $k[G]$ denotes the usual group
algebra of a group $G$.

\begin{proposition}\prlabel{graduari}
Let $G$ be an abelian group and $\mathfrak{h}$ a finite
dimensional Leibniz algebra. Then there exists a bijection between
the set of all $G$-gradings on $\mathfrak{h}$ and the set of all
bialgebra homomorphisms ${\mathcal A} (\mathfrak{h}) \to k[G]$.

The bijection is given such that the $G$-grading on $\mathfrak{h}
= \oplus_{\sigma \in G} \, \mathfrak{h}_{\sigma}^{(\theta)} $
associated to a bialgebra map $\theta: {\mathcal A} (\mathfrak{h})
\to k[G]$ is given by:
\begin{equation}\eqlabel{gradass}
\mathfrak{h}_{\sigma}^{(\theta)} := \{ x \in  \mathfrak{h} \, | \,
\bigl({\rm Id}_{\mathfrak{h}} \ot \theta \bigl) \, \circ \,
\eta_{\mathfrak{h}} (x) = x \ot \sigma  \}
\end{equation}
for all $\sigma \in G$.
\end{proposition}

\begin{proof} By applying \thref{univbialg} for the
commutative bialgebra $B := k[G]$ yields a bijection between the
set of all bialgebra homomorphisms ${\mathcal A} (\mathfrak{h})
\to k[G]$ and the set of all Leibniz algebra homomorphisms $f
\colon \mathfrak{h} \to \mathfrak{h} \otimes k[G]$ which makes
$\mathfrak{h}$ into a right $k[G]$-comodule. The proof is finished
if we show that the latter set is in bijective correspondence with
the set of all $G$-gradings on $\mathfrak{h}$.

Indeed, it is a well known fact in Hopf algebra theory
\cite[Excercise 3.2.21]{radford} that there exists a bijection
between the set of all right $k[G]$-comodule structures $f:
\mathfrak{h} \to \mathfrak{h} \ot k[G]$ on the vector space
$\mathfrak{h}$ and the set of all vector space decompositions
$\mathfrak{h} = \oplus_{\sigma \in G} \, \mathfrak{h}_{\sigma}$.
The bijection is given such that $x_{\sigma} \in
\mathfrak{h}_{\sigma}$ if and only if $f (x_{\sigma}) = x_{\sigma}
\ot \sigma$, for all $\sigma \in G$. The only thing left to prove
is that under this bijection a right coaction $f: \mathfrak{h} \to
\mathfrak{h} \ot k[G]$ is a Leibniz algebra homomorphism if and
only if $\left [\mathfrak{h}_{\sigma}, \, \mathfrak{h}_{\tau}
\right] \subseteq \mathfrak{h}_{\sigma \tau}$, for all $\sigma$,
$\tau \in G$. Indeed, let $\sigma$, $\tau \in G$ and
$x_{\sigma} \in \mathfrak{h}_{\sigma}$, $x_{\tau} \in
\mathfrak{h}_{\tau}$; then $\left [f(x_{\sigma}), \, f(x_{\tau})
\right] = \left [x_{\sigma}\ot \sigma, \, x_{\tau} \ot \tau
\right] = \left [x_{\sigma}, \, x_{\tau} \right] \ot \sigma \tau$.
Thus, we obtain that $ f (\left[ x_{\sigma}, \, x_{\tau} \right])
= \left [f(x_{\sigma}), \, f(x_{\tau}) \right]$ if and only if $
\left[x_{\sigma}, \, x_{\tau} \right] \in \mathfrak{h}_{\sigma
\tau}$. Hence, $f: \mathfrak{h} \to \mathfrak{h} \ot k[G]$ is a
Leibniz algebra homomorphism if and only if $\left
[\mathfrak{h}_{\sigma}, \, \mathfrak{h}_{\tau} \right] \subseteq
\mathfrak{h}_{\sigma \tau}$, for all $\sigma$, $\tau \in G$ and
the proof is now finished.
\end{proof}

Our next result classifies all $G$-gradings on a given
Leibniz algebra $\mathfrak{h}$, where $G$ is an abelian group. Recall that two $G$-gradings $\mathfrak{h} = \oplus_{\sigma \in G}
\, \mathfrak{h}_{\sigma} = \oplus_{\sigma \in G} \,
\mathfrak{h}_{\sigma} ^{'}$ on $\mathfrak{h}$ are called
\emph{isomorphic} if there exists $w \in {\rm Aut}_{{\rm Lbz}}
(\mathfrak{h})$ an automorphism of $\mathfrak{h}$ such that $w
(\mathfrak{h}_{\sigma}) \subseteq \mathfrak{h}_{\sigma} ^{'}$, for
all $\sigma \in G$. Since $w$ is bijective and $\mathfrak{h}$ is
$G$-graded we can prove that the last condition is equivalent to
$w (\mathfrak{h}_{\sigma}) =  \mathfrak{h}_{\sigma} ^{'}$, for all
$\sigma \in G$, which is the condition that usually appears in the
literature in the classification of $G$-gradings (\cite{eld}).
Indeed, let $x_{\sigma} ^{'} \in \mathfrak{h}_{\sigma} ^{'}$;
since $w$ is surjective there exists $y = \sum_{i=1}^t \,
y_{\tau_i} \in \mathfrak{h}$ such that $x_{\sigma} ^{'} = w(y) =
\sum_{i=1}^t \, w(y_{\tau_i})$, where $y_{\tau_i} \in
\mathfrak{h}_{\tau_i}$, for all $i = 1, \cdots, t$ are the
homogeneous components of $y$. Since $w(y_{\tau_i}) \in
\mathfrak{h}_{\tau_i}^{'}$ and $\mathfrak{h} = \oplus_{\sigma \in
G} \, \mathfrak{h}_{\sigma} ^{'}$ we obtain that  $w(y_{\tau_i}) =
0$, for all $\tau_i \neq \sigma$. As $w$ is injective, it follows that $y_{\tau_i} = 0$, for all $\tau_i \neq \sigma$; hence
$y = y_{\sigma}$ and $x_{\sigma} ^{'} = w (y_{\sigma}) \in w
(\mathfrak{h}_{\sigma})$, as needed.

We recall one more elementary fact from Hopf algebra theory: if
$H$ and $L$ are two bialgebras over a field $k$ then the abelian
group ${\rm Hom} (H, \, L)$ of all $k$-linear maps is an unital
associative algebra under the convolution product (\cite{Sw}):
$(\theta_1 \star \theta_2) (h) := \sum \, \theta_1 (h_{(1)})
\theta_2 (h_{(2)})$, for all $\theta_1$, $\theta_2 \in {\rm Hom}
(H, \, L)$ and $h\in H$.

\begin{definition}\delabel{conjug}
Let $G$ be an abelian group and $\mathfrak{h}$ a finite
dimensional Leibniz algebra. Two homomorphisms of bialgebras $\theta_1,
\theta_2: {\mathcal A} (\mathfrak{h}) \to k[G]$ are called
\emph{conjugate}, if there
exists $g \in U\bigl (G\bigl( {\mathcal A} (\mathfrak{h})^{\rm o}
\bigl)\bigl)$ an invertible group-like element of the finite dual
${\mathcal A} (\mathfrak{h})^{\rm o}$ such that $\theta_2 = g
\star \theta_1 \star g^{-1}$, in the convolution algebra ${\rm
Hom} \bigl( {\mathcal A} (\mathfrak{h}) , \, k[G] \bigl)$. We use the notation $\theta_1 \approx \theta_2$ to designate two conjugate homomorphisms.
\end{definition}

We denote by ${\rm Hom}_{\rm BiAlg} \, \bigl( {\mathcal A}
(\mathfrak{h}) , \, k[G] \bigl)/\approx $ the quotient of the
set of all bialgebra homomorphisms ${\mathcal A} (\mathfrak{h}) \to k[G]$
by the above equivalence relation and let $\hat{\theta}$ denote the
equivalence class of $\theta \in {\rm Hom}_{\rm BiAlg} \, \bigl(
{\mathcal A} (\mathfrak{h}) , \, k[G] \bigl)$. The next theorem
classifies all $G$-gradings on $\mathfrak{h}$.

\begin{theorem} \thlabel{nouaclas}
Let $G$ be an abelian group, $\mathfrak{h}$ a finite dimensional
Leibniz algebra and consider $G$-${\rm \textbf{gradings}}(\mathfrak{h})$ to be the
set of isomorphism classes of all $G$-gradings on $\mathfrak{h}$.
Then the map
$$
{\rm Hom}_{\rm BiAlg} \, \bigl( {\mathcal A} (\mathfrak{h}) , \,
k[G] \bigl)/\approx  \,\,\, \mapsto \,\, G{\rm-\textbf{gradings}}
(\mathfrak{h}), \qquad \hat{\theta} \mapsto
\mathfrak{h}^{(\theta)} := \oplus_{\sigma \in G} \,
\mathfrak{h}_{\sigma}^{(\theta)}
$$
where $\mathfrak{h}_{\sigma}^{(\theta)} = \{ x \in  \mathfrak{h}
\, | \, \bigl({\rm Id}_{\mathfrak{h}} \ot \theta \bigl) \, \circ
\, \eta_{\mathfrak{h}} (x) = x \ot \sigma  \}$, for all $\sigma
\in G$, is bijective.
\end{theorem}

\begin{proof}
Let $\{e_1, \cdots, e_n\}$ be a basis in $\mathfrak{h}$. By
\prref{graduari} for any $G$-grading $\mathfrak{h}
= \oplus_{\sigma \in G} \, \mathfrak{h}_{\sigma}$ on
$\mathfrak{h}$ there exists a unique bialgebra homomorphism $\theta:
{\mathcal A} (\mathfrak{h}) \to k[G]$ such that
$\mathfrak{h}_{\sigma} = \mathfrak{h}_{\sigma}^{(\theta)}$, for
all $\sigma \in G$. It remains to investigate when two such $G$-gradings
are isomorphic. Let $\theta_1$, $\theta_2 : {\mathcal A}
(\mathfrak{h}) \to k[G]$ be two bialgebra homomorphisms and let
$\mathfrak{h} = \mathfrak{h}^{(\theta_1)} := \oplus_{\sigma \in G}
\, \mathfrak{h}_{\sigma}^{(\theta_1)} = \oplus_{\sigma \in G} \,
\mathfrak{h}_{\sigma}^{(\theta_2)} =: \mathfrak{h}^{(\theta_2)}$
be the associated $G$-gradings. It follows from the proof of
\prref{graduari} that defining a $G$-grading on $\mathfrak{h}$ is
equivalent (and the correspondence is bijective) to defining a right
$k[G]$-comodule structure $\rho :\mathfrak{h} \to \mathfrak{h} \ot
k[G]$ on $\mathfrak{h}$ such that the right coaction $\rho$ is a Leibniz algebra homomorphism. Moreover, the right coactions
$\rho^{(\theta_1)}$ and $\rho^{(\theta_2)} : \mathfrak{h} \to
\mathfrak{h} \ot k[G]$ are implemented from $\theta_1$ and
$\theta_2$ using \thref{univbialg}, that is, they are given for any
$j = 1, 2$ by
\begin{equation} \eqlabel{3000}
\rho^{(\theta_j)} : \mathfrak{h} \to \mathfrak{h} \ot k[G], \qquad
\rho^{(\theta_j)} (e_i) = \sum_{s=1}^n \, e_s \ot \theta_j
(x_{si})
\end{equation}
for all $i = 1, \cdots, n$. Now a well
known result in Hopf algebra theory states that the two $G$-gradings
$\mathfrak{h}^{(\theta_1)}$ and $\mathfrak{h}^{(\theta_2)}$ are
isomorphic if and only if $(\mathfrak{h}, \, \rho^{(\theta_1)})$
and $(\mathfrak{h}, \, \rho^{(\theta_2)})$ are isomorphic as
Leibniz algebras and right $k[G]$-comodules, that is there exists
$w : \mathfrak{h} \to \mathfrak{h}$ an automorphism of
$\mathfrak{h}$ such that $\rho^{(\theta_2)} \, \circ w = (w \ot
{\rm Id}_{k[G]}) \, \circ \rho^{(\theta_1)}$. We apply now
\thref{automorf}: for any Leibniz algebra automorphism $w :
\mathfrak{h} \to \mathfrak{h}$ there exists a unique invertible group-like element of the finite dual $g \in U\bigl
(G\bigl( {\mathcal A} (\mathfrak{h})^{\rm o} \bigl)\bigl)$ such that $w = w_g$ is given for any $i =
1, \cdots, n$ by
\begin{equation} \eqlabel{3001}
w_g (e_i) = \sum_{s=1}^n \, g(x_{si}) \, e_s
\end{equation}
Using \equref{3000} and \equref{3001} we can easily compute that:
$$
\bigl (\rho^{(\theta_2)} \, \circ w_g \bigl) (e_i) = \sum_{a=1}^n
\, e_a \ot \bigl( \sum_{s=1}^n \, \theta_2 (x_{as}) g(x_{si})
\bigl)
$$
and
$$
\bigl( (w_g \ot {\rm Id}_{k[G]}) \, \circ \rho^{(\theta_1)} \bigl)
(e_i) = \sum_{a=1}^n \, e_a \ot  \bigl(\sum_{s=1}^n \, g(x_{as})
\theta_1 (x_{si}) \bigl)
$$
for all $i = 1, \cdots, n$. Thus, the Leibniz algebra automorphism
$w_g : \mathfrak{h} \to \mathfrak{h}$ is also a right
$k[G]$-comodule map if and only if
\begin{equation} \eqlabel{3002}
\sum_{s=1}^n \, g(x_{as}) \theta_1 (x_{si}) = \sum_{s=1}^n \,
\theta_2 (x_{as}) g(x_{si})
\end{equation}
for all $a$, $i = 1, \cdots, n$. Taking into account the formula
of the comultiplication on the universal algebra ${\mathcal A}
(\mathfrak{h})$, the equation \equref{3002} can be easily rephrased as $(g \star \theta_1) ( x_{ai}) = (\theta_2 \star
g ) ( x_{ai})$, for all $a$, $i = 1, \cdots, n$ in the convolution
algebra ${\rm Hom} \bigl( {\mathcal A} (\mathfrak{h}) , \, k[G]
\bigl)$, or (since $\{x_{ai}\}_{a, i = 1, \cdots, n}$ is a system
of generators of ${\mathcal A} (\mathfrak{h})$) just as $g \star
\theta_1 = \theta_2 \star g$. We also note that $g: {\mathcal A}
(\mathfrak{h}) \to k$ is an invertible element in the above convolution algebra.

In conclusion, we have proved that two $G$-gradings $
\mathfrak{h}^{(\theta_1)}$ and $\mathfrak{h}^{(\theta_2)}$ on
$\mathfrak{h}$ associated to two bialgebra homomorphisms $\theta_1$,
$\theta_2 : {\mathcal A} (\mathfrak{h}) \to k[G]$ are isomorphic
if and only if there exists $g \in U\bigl (G\bigl( {\mathcal A}
(\mathfrak{h})^{\rm o} \bigl)\bigl)$ such that $\theta_2 = g \star
\theta_1 \star g^{-1}$, in the convolution algebra ${\rm Hom}
\bigl( {\mathcal A} (\mathfrak{h}) , \, k[G] \bigl)$, that is
$\theta_1 \approx  \theta_2$ and the
proof is now finished.
\end{proof}

Recall that an \emph{action as automorphisms of a group $G$ on
a Leibniz algebra  $\mathfrak{h}$} is a group homomorphism
$\varphi: G \to {\rm Aut}_{{\rm Lbz}} (\mathfrak{h})$.
We give now the last application of
the universal bialgebra ${\mathcal A} (\mathfrak{h})$.

\begin{proposition} \prlabel{actiuni}
Let $G$ be a finite group and $\mathfrak{h}$ a finite dimensional
Leibniz algebra with basis $\{e_1, \cdots, e_n\}$. Then there
exists a bijection between the set of all actions as automorphisms
of $G$ on $\mathfrak{h}$ and the set of all bialgebra
homomorphisms ${\mathcal A} (\mathfrak{h}) \to k[G]^*$.

The bijection is given such that the group homomorphism
$\varphi_{\theta} : G \to {\rm Aut}_{{\rm Lbz}} (\mathfrak{h})$
associated to a bialgebra homomorphism $\theta: {\mathcal A}
(\mathfrak{h}) \to k[G]^*$ is defined as follows:
\begin{equation}\eqlabel{actiuniexp}
\varphi_{\theta} ( g) (e_i) = \sum_{s=1}^n \, <\theta(x_{si}), \,
g> \, e_s
\end{equation}
for all $g\in G$ and $i = 1, \cdots, n$.
\end{proposition}

\begin{proof} Applying \thref{univbialg} for the
commutative bialgebra $B := k[G]^*$ gives a bijection between the
set of all bialgebra homomorphisms ${\mathcal A} (\mathfrak{h})
\to k[G]^*$ and the set of all Leibniz algebra homomorphisms $f
\colon \mathfrak{h} \to \mathfrak{h} \otimes k[G]^*$ which make
$\mathfrak{h}$ into a right $k[G]^*$-comodule. The proof is
finished if we show that the latter set is in bijective
correspondence with the set of all group homomorphisms $G \to {\rm
Aut}_{{\rm Lbz}} (\mathfrak{h})$. This follows by a standard
argument in Hopf algebra theory, similar to the one used in
\cite[Lemma 1]{rad2}. We indicate very briefly how the argument
goes, leaving the details to the reader. Indeed, the category of
right $k[G]^*$-comodules is isomorphic to the category of left
$k[G]$-modules. The left action $\bullet : k[G] \ot \mathfrak{h}
\to \mathfrak{h}$ of the group algebra $k[G]$ on $\mathfrak{h}$
associated to a right coaction $f \colon \mathfrak{h} \to
\mathfrak{h} \otimes k[G]^*$ is given by $g \bullet x := \, <
x_{<1>} , \, g> \, x_{<0>}$, where we used the $\sum$-notation for
comodules, $f(x) = x_{<0>} \ot x_{<1>} \in \mathfrak{h} \otimes
k[G]^*$ (summation understood). We associate to the action
$\bullet$ the map $\varphi_{\bullet} : G \to {\rm Aut}_k
(\mathfrak{h})$, $\varphi_{\bullet} (g) (x) := g \bullet x$, for
all $g \in G$ and $x \in \mathfrak{h}$. Now, it can be easily
checked that $f \colon \mathfrak{h} \to \mathfrak{h} \otimes
k[G]^*$ being a Leibniz algebra homomorphism is equivalent to
$\varphi_{\bullet} (g)$ being an automorphism of the Leibniz
algebra $\mathfrak{h}$, for all $g \in G$ and the proof is
finished.
\end{proof}

\end{document}